\documentclass{article}

\usepackage{amsmath,amssymb,amsthm,mathrsfs,bbm}
\usepackage[colorlinks=true]{hyperref}
\usepackage{graphicx}
\usepackage{pdfsync}

\newtheorem{theorem}{Theorem}
\newtheorem{definition}{Definition}
\newtheorem{proposition}{Proposition}
\newtheorem{lemma}{Lemma}

\newtheorem{remark}{Remark}

\newcommand{\fonction}[5]{\begin{array}[t]{lrcl} #1 : & #2 &\rightarrow & #3 \\ & #4 & \mapsto & #5 \end{array}}

\DeclareMathOperator*{\argmax}{arg\,max}

\def\bs{\backslash}
\def\di{\displaystyle}
\def\eps{\varepsilon}
\def\t{\tau}

\def\qqquad{\qquad \quad}

\def\R{\mathbb{R}}
\def\N{\mathbb{N}}
\def\T{\mathbb{T}}

\def\NN{\mathrm{N}}
\def\B{\mathrm{B}}
\def\BB{\overline{\B}}
\def\C{\mathrm{C}}
\def\D{\mathrm{D}}

\def\AC{\mathrm{AC}}
\def\PC{\mathrm{PC}}
\def\L{\mathrm{L}}
\def\U{\mathrm{U}}
\def\E{\mathrm{E}}

\def\AA{\mathcal{A}}
\def\DD{\mathcal{D}}
\def\VV{\mathcal{V}}
\def\UU{\mathcal{U}}
\def\UUadm{\mathcal{U}_{\mathrm{adm}}}

\def\PP{\mathcal{P}}

\def\SS{\mathcal{S}}
\def\HH{\mathcal{H}}
\def\KK{\mathcal{K}}
\def\CC{\mathcal{C}}

\begin{document}

\title{Convergence in nonlinear optimal sampled-data control problems}

\author{Lo\"ic Bourdin and Emmanuel Tr\'elat
\thanks{L. Bourdin is with University of Limoges, XLIM Research Institute, F-87000 Limoges, France, loic.bourdin@unilim.fr.}
\thanks{E. Tr\'elat is with Sorbonne Universit\'e, CNRS, Universit\'e de Paris, Inria, Laboratoire Jacques-Louis Lions (LJLL), F-75005 Paris, France, emmanuel.trelat@sorbonne-universite.fr.}
}

\date{}
\maketitle

\begin{abstract}
Consider, on the one part, a general nonlinear finite-dimensional optimal control problem and assume that it has a unique solution whose state is denoted by~$x^*$. 
On the other part, consider the sampled-data control version of it. Under appropriate assumptions, we prove that the optimal state of the sampled-data problem converges uniformly to $x^*$ as the norm of the corresponding partition tends to zero. Moreover, applying the Pontryagin maximum principle to both problems, we prove that, if $x^*$ has a unique weak extremal lift with a costate $p$ that is normal, then the costate of the sampled-data problem converges uniformly to $p$.
In other words, under a nondegeneracy assumption, control sampling commutes, at the limit of small partitions, with the application of the Pontryagin maximum principle.
\end{abstract}



\section{Introduction}

Optimal control theory is a topic of mathematics seeking the best possible action (or \textit{control}) for steering a given dynamical system to a final configuration, while optimizing a given criterion. The milestone is certainly the Pontryagin maximum principle~\cite{pontryagin1962} which establishes a first-order necessary condition for optimality: if a trajectory is optimal, then it must be the projection onto the state space of a so-called \textit{extremal}, consisting of the state but also of a costate satisfying an adjoint equation. Actually, state and costate satisfy Hamiltonian equations, and moreover the optimal control maximizes pointwisely the Hamiltonian along the extremal. In general this maximization condition leads to express the optimal control as a function of the state and the costate.
In this classical situation where the controls are measurable functions of the time $t$ (possibly subject to some constraints), we speak of \emph{permanent controls}, in the sense that the control value can be modified at any time.

Of course, permanent controls are a mathematical model which expresses an idealized situation. In practice, when acting on concrete models or even with a computer, one obviously cannot act at any time, because controls are digital or may be frozen over a certain (even small) horizon of time. We speak then of \emph{nonpermanent controls}. A situation of interest, which has been much considered in the literature, is that of \emph{sampled-data controls}, standing for functions that can be modified only at the so-called \emph{sampling times} of a given subdivision of the interval of time (see, e.g., \cite{acker, fada, Iser, land, nesi, raga}).

Optimal sampled-data control theory has been developed in a number of contributions, such as \cite{azhm, bini2009, bini2014, chen, levis1971, souz, toiv, tou}, but this is only recently that the Pontryagin maximum principle has been established in full generality for optimal sampled-data control problems (see \cite{BT2015}, see also~\cite{BT2013,BT2016} for versions on time scales).
As for fully discrete optimal control problems, the maximization condition fails in general and must be replaced with a weaker condition. This difference with the permanent case creates a kind of non-uniformity that causes difficulties in the following expected convergence.

It is natural to expect that, as the norm of the partition (maximal distance between two successive sampling times) tends to zero, the optimal sampled-data control problem converges in some sense to the optimal permanent control problem. In this paper, our objective is to establish this kind of~$\Gamma$-convergence property to a wide extent. Under appropriate assumptions, not only we prove that the optimal state of the sampled-data problem converges uniformly to the optimal state of the permanent problem, but we also establish the uniform convergence of the costates coming from the application of the Pontryagin maximum principle to both problems. The latter property is particularly important in view of justifying strong convergence of the controls, and in view of initializing successfully numerical methods.

Similar convergence results have been obtained in~\cite{BT2017} in the
unconstrained linear-quadratic context. In the present work, we investigate the fully general nonlinear case, under control and terminal state constraints. This framework is significantly more involved and leads us to develop a Filippov-type approach.

\section{Framework and preliminaries}\label{sec2}
Throughout the paper, two positive integers $m$, $n \in \N^*$ are fixed, as well as a positive real number $T > 0$. We denote by:
\begin{itemize}
\item[--] $\C := \C( [0,T],\R^n)$ the Banach space of continuous functions defined on $[0,T]$ with values in $\R^n$, endowed with the uniform norm $\Vert \cdot \Vert_{\C}$;
\item[--] $\AC := \AC( [0,T],\R^n)$ the subspace of $\C$ of absolutely continuous functions;
\item[--] $\L^r := \L^r( [0,T],\R^m)$ the Lebesgue space of power $r$ integrable functions defined on $[0,T]$ with values in~$\R^m$, endowed with the norm $\Vert \cdot \Vert_{\L^r}$, for any~$r\in[1,+\infty]$.
\end{itemize}
A {partition} of the interval $[0,T]$ is a set $\T = \{ t_i \}_{i=0,\ldots,N}$ of real numbers satisfying $ 0 = t_0 < t_1 < \cdots < t_{N-1} < t_N = T $, where $N \in \N^*$ is a positive integer, and the norm of~$\T$ is defined by~$ \Vert \T \Vert := \max_{i=0,\ldots,N-1} \vert t_{i+1} - t_i \vert $. In the sequel $\PP$ stands for the set of all partitions of the interval~$[0,T]$.

The set $\PC^\T := \PC^\T( [0,T],\R^m)$ of all piecewise constant functions defined on~$[0,T]$ with values in $\R^m$, according to a partition $\T = \{ t_i \}_{i=0,\ldots,N} \in \PP$, is defined by 
\begin{multline*}
\PC^\T := \{ u \in \L^\infty \mid \forall i \in \{ 0,\ldots,N-1 \}, 
\qquad \exists u_i \in \R^m, \; u(t)=u_i \text{ a.e.\ } t \in [t_i,t_{i+1}) \}. 
\end{multline*}
Given any subset $\SS$ of $\R^m$, we denote by $\L^r_\SS$ (resp., $\PC^\T_\SS$) the subset of~$\L^r$ (resp., $\PC^\T$) of functions with values in $\SS$. Finally we denote by~$\langle \cdot , \cdot \rangle_n$ the Euclidean scalar product in~$\R^n$.

\subsection{Optimal permanent control problem}\label{secOCP}
Let $x_0$, $x_T \in \R^n$ be arbitrary (terminal conditions) and let $\U$ be a nonempty closed convex subset of~$\R^m$ (control constraint set).
We consider the optimal permanent control problem
\begin{equation}\label{theproblem}\tag{OCP}
\begin{array}{ll}
\text{min} & \CC (x,u) := \di \int_0^T L (x(s),u(s),s) \; ds  \\
& \\[-7pt]
\text{s.t.} & \left\lbrace \begin{array}{l}
x \in \AC , \; u \in \L^\infty_\U , \\[3pt]
\dot{x}(t) = f ( x(t),u(t),t ) \text{ a.e.\ } t \in [0,T] , \\[3pt]
x(0) = x_0 , \; x(T) = x_T ,
\end{array} \right.
\end{array} 
\end{equation}
where $x$ is the state and $u$ is the control.
We assume that the {dynamics} $f : \R^n \times \R^m \times [0,T] \to \R^n$ and the {Lagrange function} $L : \R^n \times \R^m \times [0,T] \to \R$ are of class~$\C^2$. In~\eqref{theproblem}, the control is said to be \emph{permanent}, in the sense that its value can be modified at any real time~$t \in [0,T]$. A pair~$ (x,u )$ is said to be \emph{admissible} for~\eqref{theproblem} if it satisfies all the constraints of~\eqref{theproblem}. Finally we denote by~$\AA$ the set of all admissible pairs for~\eqref{theproblem}.

The epigraph of the extended velocity set is defined by
$$ \VV (x,t) := \{ (f(x,u,t),L(x,u,t)+\gamma) \mid (u,\gamma) \in \U \times \R_+ \} $$
for all $(x,t) \in \R^n \times [0,T]$. According to the Filippov theorem~\cite{filippov1959}, if $\AA \neq \emptyset$, under the {compactness assumption} 
\begin{equation}\label{eqhcomp}\tag{$\mathrm{H}^{\textrm{comp}}$}
\left\lbrace \begin{array}{l}
\U \text{ is compact}, \\[3pt]
\exists R > 0, \; \forall (x,u) \in \AA, \; \Vert x \Vert_{\C} \leq R,
\end{array} \right.
\end{equation}
and under the convexity assumption
\begin{equation}\label{eqhconv}\tag{$\mathrm{H}^{\textrm{conv}}$}
\forall (x,t) \in \R^n \times [0,T], \; \VV (x,t) \text{ is convex}, 
\end{equation}
then \eqref{theproblem} has at least one solution (see~\cite{Cesari,trelat2005} and Appendix~\ref{appfilippov}).

The \emph{Hamiltonian}~$H : \R^n \times \R^m \times \R^n \times \R \times [0,T] \to \R$ associated with \eqref{theproblem} is defined by
$$ H(x,u,p,p^0,t) := \langle p,f(x,u,t) \rangle_{n} +p^0 L (x,u,t) $$
for all $(x,u,p,p^0,t) \in \R^n \times \R^m \times \R^n \times \R \times [0,T]$. Given~$(x,u) \in \AA$, a nontrivial pair $(p,p^0) \in \AC \times \R_-$ is said to be a \emph{weak extremal lift} of $(x,u)$ if the adjoint equation 
\begin{equation}\tag{AE}\label{eqAE}
\!\!\! \dot{p}(t) = - \nabla_x f (x(t),u(t),t)^\top p(t) 
- p^0 \nabla_x L (x(t),u(t),t)
\end{equation}
and the Hamiltonian gradient condition 
\begin{equation}\tag{HG}\label{eqHG}
 \nabla_u H(x(t),u(t),p(t),p^0,t) \in \NN_\U [u(t)] 
\end{equation}
are satisfied for almost every $t \in [0,T]$, where $\NN_\U [u(t)]$ is the normal cone to $\U$ at~$u(t)$. A nontrivial pair $(p,p^0) \in \AC \times \R_-$ is said to be a \emph{strong extremal lift} of~$(x,u)$ if \eqref{eqAE} and the Hamiltonian maximization condition (stronger than \eqref{eqHG}) 
\begin{equation}\tag{HM}\label{eqHM}
u(t) \in \argmax_{\omega \in \U} H(x(t),\omega,p(t),p^0,t) 
\end{equation}
are satisfied for almost every $t \in [0,T]$. 

\begin{remark}\label{rempermanent}
Let $(p,p^0)$ be a (weak or strong) extremal lift of a pair~$(x,u)\in\AA$. Recall that~$p$ is usually called costate (or adjoint vector), and that 
\begin{equation*}
\left\lbrace \begin{array}{l}
\dot{x}(t) = \nabla_p H (x(t),u(t),p(t),p^0,t), \\[3pt]
\dot{p}(t) = - \nabla_x H (x(t),u(t),p(t),p^0,t),
\end{array} \right.
\end{equation*}
for almost every $t \in [0,T]$. The pair~$(p,p^0)$ is defined up to (constant in time) positive scaling. In the normal case~$p^0 \neq 0$, it is standard to normalize the pair~$(p,p^0)$ so that $p^0 = - 1$. In the abnormal case~$p^0 = 0$, by homogeneity of~\eqref{eqAE} and Cauchy uniqueness, we must have $p(T) \neq 0_{\R^n}$.
\end{remark}

According to the Pontryagin maximum principle~\cite{pontryagin1962}, if~$(x^*,u^*)$ is a solution to \eqref{theproblem}, then it has at least one strong extremal lift~$(p,p^0)$.

\begin{remark}
Recall that a weak extremal lift does not necessarily coincide with a strong extremal lift (even if it corresponds to a solution to \eqref{theproblem}). Let us be more precise.
Of course, any strong extremal lift of an admissible pair is also a weak extremal lift.
The converse is true if $H$ is concave with respect to $u$, but fails to be true in general.
As a counterexample, take~$T=n=m=1$, $x_0 = x_T = 0$, $\U = [-1,1]$, $f(x,u,t) = u^3$ and $L (x,u,t) = 0$ for all $(x,u,t) \in \R \times \R \times [0,1]$. Then consider~$p^0 \in \{ -1 , 0 \}$, $x(t)=u(t)=0$ and~$p(t)=1$ for every~$t \in [0,1]$. Then $(p,p^0)$ is a weak extremal lift of the admissible pair~$(x,u)$ but is not a strong extremal lift. Note that the admissible pair~$(x,u)$ is a solution to~\eqref{theproblem} and that the weak extremal lift can be either normal or abnormal. We refer to~\cite{BT2021} for a detailed discussion.
\end{remark}

\subsection{Optimal sampled-data control problems}\label{secOCPT}

Given any partition $\T = \{ t_i \}_{i=0,\ldots,N} \in \PP$, we consider the optimal sampled-data control problem
\begin{equation}\label{thesampledproblem}\tag{OCP${}_\T$}
\begin{array}{rl}
\text{min} &\CC (x,u) := \di \int_0^T L (x(s),u(s),s) \; ds \\
& \\[-7pt]
\text{s.t.} & \!\!\! \left\lbrace \begin{array}{l}
x \in \AC, \; u \in \PC^\T_\U, \\[3pt]
\dot{x}(t) = f ( x(t),u(t),t ) \text{ a.e.\ } t \in [0,T], \\[3pt]
x(0) = x_0 , \; x(T) = x_T.
\end{array} \right.
\end{array} 
\end{equation}
In contrast to \eqref{theproblem}, the control in \eqref{thesampledproblem} is \emph{nonpermanent}, in the sense that its value can be modified only at the sampling times~$t_i$. In what follows, we denote by~$\AA_\T$ the set of all admissible pairs for~\eqref{thesampledproblem}. Note that~$\AA_\T \subset \AA$.

Since $\PC^\T$ is a finite-dimensional subspace of~$\L^\infty$, the following Filippov-type existence result does not require any convexity assumption, in contrast to the permanent control case (since we have not found this result in the literature, a proof is provided in Appendix~\ref{appmainresults}).

\begin{proposition}\label{propsampledfilippov}
If $\AA_\T \neq \emptyset$ for some~$\T \in \PP$, under \eqref{eqhcomp}, \eqref{thesampledproblem} has at least one solution. 
\end{proposition}

Given a pair~$(x,u) \in \AA_\T$, a nontrivial pair $(p,p^0) \in \AC \times \R_-$ is said to be a \emph{$\T$-averaged weak extremal lift} of~$(x,u)$ if it satisfies the adjoint equation~\eqref{eqAE} and the $\T$-averaged Hamiltonian gradient condition
\begin{equation}\tag{AHG${}_\T$}\label{eqAHG}
 \int_{t_i}^{t_{i+1}}  \nabla_u H(x(s),u_i,p(s),p^0,s) \; ds \in \NN_\U [u_i]
\end{equation}
for every $i \in \{ 0,\ldots,N-1 \}$. Remark~\ref{rempermanent} applies as well to the concept of $\T$-averaged weak extremal lift.

According to the Pontryagin maximum principle obtained recently in~\cite{BT2015,BT2016}, if~$(x^*_\T,u^*_\T)$ is a solution to \eqref{thesampledproblem}, then it has at least one $\T$-averaged weak extremal lift $(p_\T,p^0_\T)$.

\section{Main result and comments}\label{subsecmainresults}
The objective of this paper is to establish convergence of solutions to \eqref{thesampledproblem} to solutions to \eqref{theproblem} as~$\Vert \T \Vert\rightarrow 0$. 

Problem~\eqref{theproblem} (resp., Problem~\eqref{thesampledproblem} for some~$\T \in \PP$) can be formulated as the problem of minimizing the functional~$\CC$ over~$\AA$ (resp., $\AA_\T$). Of course, we have assume that $\AA\neq\emptyset$, i.e., that the target point $x_T$ is reachable from the initial point~$x_0$ with a (permanent) $\L^\infty_\U$-control. A first fundamental question is to know whether, for $\Vert \T \Vert$ small enough, one has $\AA_T\neq\emptyset$, i.e., whether $x_T$ can be reached from $x_0$ with a (sampled-data) $\PC^\T_\U$-control. 
This question happens to be more difficult than expected and has been chiefly investigated in our recent paper \cite{BT2021}, which can be seen as a preliminary to the present one.
Of course, if $\AA_\T \neq \emptyset$ for some~$\T \in \PP$, then~$\AA \neq \emptyset$, but the converse is not true in general, even for small values of~$\Vert \T \Vert$, as illustrated in~\cite[Example~1.1]{BT2021}.
The following instrumental reachability result comes from \cite{BT2021}.

\begin{lemma}\label{lemmaaccessible}
If a pair~$(x,u) \in \AA$ has no abnormal strong extremal lift, then there exists $\delta > 0$ such that~$\AA_\T \neq \emptyset$ for every~$\T \in \PP$ satisfying~$\Vert \T \Vert \leq \delta$. Moreover, for any~$\rho > 0$, the threshold~$\delta > 0$ can be chosen small enough to guarantee that, for every~$\T \in \PP$ satisfying~$\Vert \T \Vert \leq \delta$, the set~$\AA_\T$ contains a pair~$(x_\T,u_\T)$ such that~$\Vert u_\T - u \Vert_{\L^1} \leq \rho$.
\end{lemma} 

We introduce the convexity assumption
\begin{equation}\label{eqhconv2}\tag{$\mathrm{H}^{\textrm{conv}}_{\nabla}$}
\forall (x,t) \in \R^n \times [0,T], \; \VV_{\nabla} (x,t) \text{ is convex}, 
\end{equation}
where the epigraph of the extended ``gradient velocity set" is defined by
$$ \VV_{\nabla} (x,t) := \{ (f(x,u,t), F_\nabla(x,u,t)+\gamma ) \mid (u,\gamma) \in \U \times \Gamma \} $$
with
\begin{multline*}
F_\nabla(x,u,t) := \Big(  (\nabla_x f, \nabla_u f )(x,u,t), \nabla_u f(x,u,t) u , \\[3pt]
 L(x,u,t) , (\nabla_x L, \nabla_u L)(x,u,t) , \langle \nabla_u L(x,u,t) , u \rangle_{m} \Big) ,
\end{multline*}
for all $(x,u,t) \in \R^n \times \R^m \times [0,T]$ and
$$
\Gamma := \{ ( 0_{\R^{n \times n}} , 0_{\R^{n \times m}} ) \} \times \{ 0_{\R^n}  \} \times \R_+ \times \{ ( 0_{\R^n} , 0_{\R^m} ) \} \times \R_+ .
$$
Note that~\eqref{eqhconv2} implies~\eqref{eqhconv}.

\begin{theorem}\label{thmcv}
Assume that $\AA \neq \emptyset$. Under~\eqref{eqhcomp} and~\eqref{eqhconv}, if the solution~$(x^*,u^*)$ to~\eqref{theproblem} is unique and has no abnormal strong extremal lift, then there exists~$\delta > 0$ such that, for every~$\T \in \PP$ satisfying~$\Vert \T \Vert \leq \delta$, \eqref{thesampledproblem} {has at least one solution. Moreover, all solutions~$(x^*_\T,u^*_\T)$ to \eqref{thesampledproblem} satisfy:}
\begin{enumerate}
\item[\rm{(i)}] $x^*_\T$ converges uniformly to $x^*$ on $[0,T]$ as~$\Vert \T \Vert\rightarrow 0$;
\item[\rm{(ii)}] $\CC(x^*_\T,u^*_\T)$ converges to $\CC(x^*,u^*)$ as~$\Vert \T \Vert\rightarrow 0$.
\end{enumerate}
Under the additional assumption~\eqref{eqhconv2}, if~$(x^*,u^*)$ has a unique weak extremal lift~$(p,p^0)$ that is normal, the threshold~$\delta > 0$ can be chosen small enough to guarantee that, for every~$\T \in \PP$ satisfying~$\Vert \T \Vert \leq \delta$, any $\T$-averaged weak extremal lift~$(p_\T,p^0_\T)$ of~$(x^*_\T,u^*_\T)$ is normal. Furthermore, normalizing the extremal lifts so that~$p^0 = p^0_\T = -1$,
\begin{enumerate}
\item[\rm{(iii)}] $p_\T$ converges uniformly to $p$ on $[0,T]$ as~$\Vert \T \Vert\rightarrow 0$.
\end{enumerate}
\end{theorem}

Theorem \ref{thmcv} is proved in Appendix~\ref{appmainresults}. The proof of the second part is more involved than the one of the first part.

\begin{remark}\label{remaffinequad}
Theorem~\ref{thmcv} applies to control-affine systems with quadratic cost:
$$ f(x,u,t) = A(x,t)u+B(x,t), $$
$$ L(x,u,t) = \dfrac{1}{2} \langle R(t)u,u \rangle_m + \langle Q(x,t),u \rangle_m + S(x,t) , $$
where $A$, $B$, $Q$, $R$ and~$S$ are of class~$\C^2$ and $R(t) \in \R^{m \times m}$ is a symmetric positive semidefinite matrix for every $t \in [0,T]$, 
with~$\U$ compact and~$A$, $B$ growing at most linearly at infinity. 
Indeed, then, \eqref{eqhcomp} and~\eqref{eqhconv2} are satisfied.
\end{remark}

\begin{remark}\label{rem_CV_control}
The first part of Theorem~\ref{thmcv} establishes, under~\eqref{eqhcomp} and~\eqref{eqhconv}, convergence of the optimal trajectories and costs; but not of the optimal controls. 
The convergence of costates, established under the additional assumption~\eqref{eqhconv2}, implies the uniform convergence of optimal controls under the additional assumption that, as a consequence of the Hamiltonian maximization condition~\eqref{eqHM} (or the weaker Hamiltonian gradient condition~\eqref{eqHG}), one can express the optimal permanent control~$u^*$ as a continuous function of $(x^*,p)$:
\begin{equation}\label{eq1}
u^*(t) = \mathcal{F} ( x^*(t),p(t),t), \; \text{ a.e.}\ t\in[0,T],
\end{equation}
and similarly that, for every~$\T = \{ t_i \}_{i=0,\ldots,N} \in \PP$, as a consequence of the $\T$-averaged Hamiltonian gradient condition~\eqref{eqAHG}, one can express the values~$u^*_{\T,i}$ of the optimal sampled-data control~$u^*_\T$ as continuous functions of the restrictions to the sampling intervals~$[t_i,t_{i+1}]$ of $(x^*_\T,p_\T)$:
\begin{equation}\label{eq2}
 u^*_{\T,i} = \mathcal{F}_i \Big( x^*_{\T,|[t_i,t_{i+1}]},p_{\T,|[t_i,t_{i+1}]} \Big), \; \forall i \in \{ 0,\ldots N-1 \}. 
\end{equation}
The above \eqref{eq1} and~\eqref{eq2} are obviously true in the linear-quadratic case (see~\cite[Propositions~2 and~3]{BT2017} for details), and more generally for control-affine systems with quadratic cost (see Remark \ref{remaffinequad}).
For more general nonlinear problems, \eqref{eq1} and~\eqref{eq2} are true under appropriate Legendre-Clebsch assumptions (see \cite{BonnardChyba, BressanPiccoli, Trelat_JOTA2012}).
\end{remark}



\begin{remark}
The convergence of costates obtained in the second part of Theorem~\ref{thmcv} is important not only to infer the convergence of the optimal controls but also, from the numerical point of view, to initialize successfully shooting methods (see \cite{Trelat_JOTA2012}).
\end{remark}

\begin{remark}
Assumption \eqref{eqhconv} is usually used to derive Filippov-type existence results for optimal controls (see~\cite{Cesari, LeeMarkus}), by ensuring the uniform convergence of states (satisfying nonlinear differential equations) associated with a minimizing sequence.
The stronger assumption \eqref{eqhconv2} is not usual: it assumes the convexity, not only of the epigraphs~$\mathcal{V}(x,t)$ of the extended velocity sets, but also of the extended ones~$\mathcal{V}_\nabla(x,t)$ including gradients of~$f$ and~$L$. This assumption is required in Theorem~\ref{thmcv} to establish the uniform convergence of costates (and, as a result, of controls). Our proof is based on the convergence of variation vectors (well known in the proof of Pontryagin maximum principle) which satisfy linearized differential equations involving those gradients. This is where the convexity of the sets~$\mathcal{V}_\nabla(x,t)$ is needed, in order to follow and develop the classical Filippov approach.
\end{remark}

\begin{remark}
Theorem~\ref{thmcv} can be straightforwardly extended to the case of general terminal constraints $g(x(0),x(T)) \in \mathcal{M}$ and of free final time. The issue of adding running state constraints is let open.
\end{remark}

\begin{remark}
The absence of abnormal strong extremal lift considered in Theorem~\ref{thmcv} is, in some sense, a generic property. For control-affine systems with quadratic cost (see Remark \ref{remaffinequad}), the set of endpoints of abnormal minimizers is of empty interior or even of zero measure under appropriate assumptions (see~\cite{RiffordTrelat_MA2005, RiffordTrelat_JEMS2009}), and is empty if $m\geq 3$ for Whitney generic systems (see~\cite{ChitourJeanTrelat_JDG2006, ChitourJeanTrelat_SICON2008}).
\end{remark}

\begin{remark}
The convexity of the set~$\U$ is required in the proof of Theorem~\ref{thmcv} to apply the reachability Lemma~\ref{lemmaaccessible} (see~\cite[Theorem~1.1 and Remarks~3.5 and~3.6 for a discussion on that point]{BT2021}) and its compactness is required to apply a technical Filippov-type lemma (Lemma~\ref{lemtech1}). 
The proof of the second part of Theorem~\ref{thmcv} involves continuous second-order gradients of~$f$ and~$L$. Nevertheless the~$\C^2$-regularity of~$f$ and~$L$ can be slightly relaxed with respect to the variables~$u$ and~$t$.
\end{remark}

\begin{remark}
In order to obtain easy-to-read statements, we have assumed in Theorem~\ref{thmcv} that $(x^*,u^*)$ is the unique solution to \eqref{theproblem} (and, in the second part, that it has a unique weak extremal lift). Actually Theorem~\ref{thmcv} remains valid without any uniqueness assumption but then the statements must be written in terms of closure points, as in \cite{HaberkornTrelat_SICON2011,SilvaTrelat_TAC2010}.
\end{remark}

\appendix

\section{Technical preliminaries}\label{app1}

\subsection{Convergence in linear differential equations}

\begin{proposition}\label{propconvlineardiffeq}
Let $z_k \in \AC$ be the unique solution to the (forward) linear Cauchy problem
\begin{equation*}
\left\lbrace
\begin{array}{l}
\dot{z}_k(t) = A_k(t) z_k(t) + B_k(t) v_k(t) + C_k(t) \text{ a.e.\ } t \in [0,T] , \\[3pt]
z_k(0) = \Psi_k, 
\end{array}
\right.
\end{equation*}
where $A_k \in \L^\infty([0,T],\R^{n \times n})$, $B_k \in \L^\infty([0,T],\R^{n \times m})$, $v_k \in \L^\infty([0,T],\R^m)$, $C_k \in \L^\infty([0,T],\R^n)$ and $\Psi_k \in \R^n$, for all~$k \in \N$. 
If:
\begin{itemize}
\item $A_k$ converges weakly-star in~$\L^\infty([0,T],\R^{n \times n})$ to $A$, 
\item $B_k$ converges weakly-star in $\L^\infty([0,T],\R^{n \times m})$ to $B$, 
\item $C_k$ converges weakly-star in~$\L^\infty([0,T],\R^n)$ to $C$, 
\item $v_k$ converges in~$\L^1([0,T],\R^m)$ to $v \in \L^\infty([0,T],\R^m)$, 
\item $\Psi_k$ converges in $\R^n$ to $\Psi$, 
\end{itemize}
as $k\rightarrow+\infty$, 
then $z_k$ converges uniformly on~$[0,T]$ to~$z \in \AC$ that is the unique solution to the linear Cauchy problem
\begin{equation*}
\left\lbrace
\begin{array}{l}
\dot{z}(t) = A(t) z(t) + B(t) v(t) +C(t) \text{ a.e.\ } t \in [0,T] , \\[3pt]
z(0) = \Psi. 
\end{array}
\right.
\end{equation*}
\end{proposition}

\begin{proof}
The weak-star convergences imply that there exist~$M_A \geq 0$ and $M_B \geq 0$ such that $\Vert A_k \Vert_{\L^\infty([0,T],\R^{n \times n})} \leq M_A$ and~$\Vert B_k \Vert_{\L^\infty([0,T],\R^{n \times m})} \leq M_B$ for every~$k \in \N$. It follows from the Duhamel formula that
$$
\Vert z_k(t) - z(t) \Vert_{\R^n} \leq \Vert \Psi_k - \Psi \Vert_{\R^n} + \Vert g_k \Vert_{\C} + M_B \Vert v_k - v \Vert_{\L^1} 
+ M_A \int_0^t \Vert z_k(s) - z(s) \Vert_{\R^n} \; ds,
$$
where
$$
g_k (t) := \int_0^t (A_k(s)-A(s)) z(s)+ (B_k(s)-B(s)) v(s) 
 + (C_k(s)-C(s)) \; ds, 
$$
for all $t \in [0,T]$ and $k \in \N$. By the Gronwall lemma, we get
$$ \Vert z_k - z \Vert_{\C} \leq ( \Vert \Psi_k - \Psi \Vert_{\R^n} + \Vert g_k \Vert_{\C} + M_B \Vert v_k - v \Vert_{\L^1} ) e^{M_A T}, $$
for every~$k \in \N$. The weak-star convergences imply that the sequence $(g_k)_{k \in \N}$ converges pointwisely on~$[0,T]$ to the null function, and thus uniformly by equi-Lipschitz continuity (as in \cite[Lemma~3]{BT2017}). The proof is complete.
\end{proof}

Proposition~\ref{propconvlineardiffeq} is obviously adapted to backward linear Cauchy problems.

\subsection{Approximation by piecewise constant functions}\label{secapprox}

For all~$\T = \{ t_i \}_{i=0,\ldots,N} \in \PP$ and~$u \in \L^1$, we denote by~$u^\T \in \PC^\T$ the {averaged piecewise constant function} defined by
\begin{equation}\label{eqaveraged}
u^\T_i := \dfrac{1}{t_{i+1}-t_i} \int_{t_i}^{t_{i+1}} u(s) \; ds 
\end{equation}
for every $i \in \{0,\ldots,N-1\}$. The next two lemmas are taken from~\cite[Appendix~B]{BT2021}.

\begin{lemma}\label{lem1}
For every~$u \in \L^1$, 
$\lim_{\Vert \T \Vert \to 0} \Vert u^\T - u \Vert_{\L^1} = 0 $.
\end{lemma}

\begin{lemma}\label{lem2}
If $u \in \L^1_\U$ then~$u^\T \in \PC^\T_\U$ for every~$\T \in \PP$.
\end{lemma}

\section{Input-output maps and cost function}\label{appinput}

Given any $u \in \L^1$, we define the forward Cauchy problem
\begin{equation}\tag{CP${}_u$}\label{eqCPu}
\left\lbrace
\begin{array}{l}
\dot{x}(t) = f(x(t),u(t),t) \text{ a.e.\ } t \in [0,T] , \\[3pt]
x(0) = x_0.
\end{array}
\right.
\end{equation}

\begin{definition}
A solution to~\eqref{eqCPu} for some $u \in \L^1$ is a function~$x \in \AC$ satisfying $x(0)=x_0$ and~$\dot{x}(t) = f(x(t),u(t),t)$ for almost every $t \in [0,T]$. We define:
\begin{enumerate}
\item[\rm{-}] $\UU $ as the set of controls $u \in \L^1$ such that~\eqref{eqCPu} admits a unique solution denoted by $x(\cdot,u)$.
\item[\rm{-}] $\UUadm:=\{u \in \UU \mid  (x(\cdot,u),u) \in \AA\}$.
\end{enumerate}
\end{definition}

\begin{definition}
The maps
$$ \fonction{\E}{\UU}{\C}{u}{\E(u) := x(\cdot,u)} $$
and
$$ \fonction{\E_T}{\UU}{\R^n}{u}{\E_T(u) := x(T,u)} $$
are respectively called the \textit{input-output map} and the \textit{final input-output map}. The map
$$ \fonction{\KK}{ \DD (\KK) }{\R}{u}{ \KK(u) := \di \int_0^T L(x(s,u),u(s),s) \; ds} $$
is called the \textit{cost function}, where $\DD (\KK):=\{u \in \UU  \mid  L(x(\cdot,u),u,\cdot) \in \L^1([0,T],\R)\}$.
\end{definition}

\begin{remark}\label{rem587}
\begin{enumerate}
\item[\rm{(i)}] $\AA \neq \emptyset$ if and only if $\UUadm \neq \emptyset$;
\item[\rm{(ii)}] $ \UUadm \subset ( \UU \cap \L^\infty_\U ) \subset ( \UU \cap \L^\infty ) \subset \DD (\KK) $.
\end{enumerate}
\end{remark}

\subsection{Regularity in~$\L^\infty$-norm}\label{secinputoutput}
The following result follows from the Cauchy-Lipschitz (or Picard-Lindel\"of) theory for Carath\'eodory dynamics, in particular from the Gronwall lemma.
\begin{proposition}\label{propinputoutput}
\begin{enumerate}
\item[\rm{(i)}] $\UU \cap \L^\infty$ is an open subset of $\L^\infty$.
\item[\rm{(ii)}] The restriction $\E^\infty$ of~$\E$ to~$\UU \cap \L^\infty$ is of class~$\C^1$ in $\L^\infty$-norm, and
$$ \D\E^\infty (u)(v) = w(\cdot,u,v), $$
for all $u \in \UU \cap \L^\infty$ and $v \in \L^\infty$, where $w(\cdot,u,v) \in \AC $ is the unique solution to the linear Cauchy problem
\begin{equation*}
\left\lbrace
\begin{array}{l}
\dot{w}(t) =  \nabla_x f(x(t,u),u(t),t) w(t) \\[3pt] 
\qqquad + \nabla_u f(x(t,u),u(t),t) v(t) \; \text{ a.e.\ } t \in [0,T] , \\[3pt]
w(0) =  0_{\R^n}. 
\end{array}
\right.
\end{equation*}
\item[\rm{(iii)}] The restriction~$\E_T^\infty$ of~$\E_T$ to~$\UU \cap \L^\infty$ is of class~$\C^1$  in $\L^\infty$-norm, and
$$
 \D\E_T^\infty (u)(v) = w(T,u,v) 
 = \int_0^T \Phi_u (T,s) \nabla_u f(x(s,u),u(s),s) v(s) \; ds 
$$
for all~$u \in \UU \cap \L^\infty$ and~$v \in \L^\infty$, where~$\Phi_u (\cdot,\cdot)$ is the state-transition matrix of $\nabla_x f(x(\cdot,u),u,\cdot) \in \L^\infty([0,T],\R^{n \times n})$.
\item[\rm{(iv)}] The restriction~$\KK^\infty$ of~$\KK$ to~$\UU \cap \L^\infty$ is of class~$\C^1$ in $\L^\infty$-norm, and
\begin{multline*}
\D\KK^\infty (u)(v) = w^0(T,u,v) \\[3pt]
= \int_0^T \langle \nabla_x L(x(s,u),u(s),s) , w(s,u,v) \rangle_{n} 
+ \langle \nabla_u L(x(s,u),u(s),s) , v(s) \rangle_{m} \; ds
\end{multline*}
for all $u \in \UU \cap \L^\infty$ and $v \in \L^\infty$, where $w^0(\cdot,u,v) \in \AC([0,T],\R) $ is the unique solution to the (trivial) Cauchy problem
\begin{equation*}
\left\lbrace
\begin{array}{l}
\dot{w}^0(t) = \langle \nabla_x L(x(t,u),u(t),t) , w(t,u,v) \rangle_{n}  + \langle \nabla_u L(x(t,u),u(t),t) , v(t) \rangle_{m} \\[3pt] 
w^0(0) = 0. 
\end{array}
\right.
\end{equation*}
\end{enumerate}
\end{proposition}

The vectors~$w$ and~$w^0$ are usually called \emph{variation vectors}.
Proposition~\ref{propinputoutput} is well known in optimal control theory, but is not sufficient to prove Proposition~\ref{propsampledfilippov} and Theorem~\ref{thmcv} in Appendix~\ref{appmainresults}: we will need continuity results in~$\L^1$-norm (and not in $\L^\infty$-norm), which require truncature techniques introduced next.

\subsection{Truncation and continuity in $\L^1$-norm}\label{sectruncated}

For every $R > 0$, let $\xi^R : \R^n \times \R^m \to \R$ be a function of class $\C^2$ such that
$$
\xi^R(x,u) = \left\lbrace \begin{array}{l}
1 \text{ if } (x,u) \in \BB_{\R^n}(0_{\R^n},2R) \times \BB_{\R^m}(0_{\R^m},2R), \\[3pt]
0 \text{ if } (x,u) \notin \B_{\R^n}(0_{\R^n},3R)  \times \B_{\R^m}(0_{\R^m},3R) .
\end{array} \right.
$$
We define the \emph{truncated dynamics} $f^R : \R^n \times \R^m \times [0,T] \to \R^n$  by $ f^R (x,u,t) := \xi^R(x,u) f(x,u,t) $ for all~$(x,u,t) \in \R^n \times \R^m \times [0,T]$. Accordingly, we denote with an upper~$R$ all objects considered previously, now for the truncated dynamics~$f^R$: the Cauchy problem~(CP${}^R_u$), the set $\UU^R$, the state~$x^R(\cdot,u)$ for any $u \in \UU^R$, the input-ouput maps $\E^R$ and $\E_T^R$, the cost function~$\KK^R$ associated with the truncated Lagrange function $L^R$ defined similarly to~$f^R$, etc. 
The next result follows from the Cauchy-Lipschitz (or Picard-Lindel\"of) theory for Carath\'eodory dynamics, in particular from the Gronwall lemma.

\begin{proposition}\label{propinputoutputR}
Given any $R > 0$: 
\begin{enumerate}
\item[\rm{(i)}] $\DD(\KK^R) = \UU^R = \L^1 $;
\item[\rm{(ii)}] $\E^R$, $\E_T^R$ and $\KK^R$ are continuous in $\L^1$-norm.
\end{enumerate}
\end{proposition}

\begin{remark}\label{remtruncated}
\begin{enumerate}
\item[\rm{(i)}] Let $u \in \UU \cap \L^\infty$ and $R > 0$ be such that~$\Vert x(\cdot,u) \Vert_{\C} \leq R$ and $\Vert u \Vert_{\L^\infty} \leq R$. Then $x^R(\cdot,u) = x(\cdot,u)$.
\item[\rm{(ii)}] Conversely, let $u \in \L^\infty$ and assume that there exists~$R > 0$ such that~$\Vert x^R(\cdot,u) \Vert_{\C} \leq R$ and $\Vert u \Vert_{\L^\infty} \leq R$. Then $u \in \UU$ and~$x(\cdot,u) = x^R(\cdot,u)$.
\end{enumerate}
\end{remark}

\subsection{Characterizations  of extremal lifts}\label{seccharacterization}

\begin{lemma}\label{lemcharactweakextremal}
Let $(x,u) \in \AA$ and~$(p,p^0) \in \AC \times \R_-$ be a nontrivial pair satisfying the adjoint equation~\eqref{eqAE}. Then~$(p,p^0)$ is a weak extremal lift of~$(x,u)$ if and only if 
\begin{equation}\label{ineq01} 
\langle p(T) , w(T,u,v-u) \rangle_{n} + p^0 w^0(T,u,v-u) \leq 0 
\end{equation}
for every $v \in \L^\infty_\U$, where $w$, $w^0$ are defined in Proposition~\ref{propinputoutput}.
\end{lemma}

\begin{proof}
Given any $v \in \L^\infty_\U$, we denote by~$z_v \in \AC([0,T],\R)$ the function defined by~$z_v(t) := \langle p(t) , w(t,u,v-u) \rangle_{n} + p^0 w^0 (t,u,v-u)$ for every~$t \in [0,T]$, which satisfies $z_v(0)=0$. Inequality~\eqref{ineq01} can be rewritten as $z_v(T) \leq 0$ and  
$$ \dot{z}_v(t) = \langle \nabla_u H(x(t),u(t),p(t),p^0,t) , v(t)-u(t) \rangle_{m} $$ 
for almost every $t \in [0,T]$ and for every $v \in \L^\infty_\U$. Assume that~$(p,p^0)$ is a weak extremal lift of~$(x,u)$. Let $v \in \L^\infty_\U$. From the Hamiltonian gradient condition~\eqref{eqHG} we get that~$\dot{z}_v (t) \leq 0$ for almost every $t \in [0,T]$ and thus $z_v(T) \leq 0$, which gives the first part of the proof. Conversely, we infer from~\eqref{ineq01} that
$$
\int_0^T \langle \nabla_u H (x(s),u(s),p(s),p^0,s) , v(s)-u(s) \rangle_{m} \; ds 
= \int_0^T \dot{z}_v (s) \; ds = z_v(T)-z_v(0) = z_v (T) \leq 0
$$
for every $v \in \L^\infty_\U$. Then, for any $\omega \in \U$ and any~$\t \in [0,T)$ that is a Lebesgue point of~$\nabla_u H (x,u,p,p^0,\cdot) \in \L^\infty$ and of~$\langle \nabla_u H (x,u,p,p^0,\cdot),u \rangle_{m} \in \L^\infty([0,T],\R)$, we take $v \in \L^\infty_\U$ as the needle-like perturbation of~$u$ given by
$$ v(t) := \left\lbrace \begin{array}{lcl}
\omega & \text{if} & t \in [\t,\t+\eps), \\[3pt]
u(t) & \text{if} & t \notin [\t,\t+\eps),
\end{array}
\right. $$
for almost every $t \in [0,T]$ and all $0 < \eps \leq T-\t$. Then
$$ \dfrac{1}{\eps} \int_\t^{\t+\eps} \langle \nabla_u H (x(s),u(s),p(s),p^0,s) , \omega-u(s) \rangle_{m} \; ds \leq 0 $$
when $0 < \eps \leq T-\t$. Taking the limit $\eps \to 0^+$ gives~\eqref{eqHG}. 
\end{proof}

\begin{lemma}\label{lemcharactsampledextremal}
Let $\T \in \PP$, $(x,u) \in \AA_\T$ and~$(p,p^0) \in \AC \times \R_-$ be a nontrivial pair satisfying the adjoint equation~\eqref{eqAE}. Then~$(p,p^0)$ is a $\T$-averaged weak extremal lift of~$(x,u)$ if and only if 
\begin{equation}\label{ineq02}
\langle p(T) , w(T,u,v-u) \rangle_{n} + p^0 w^0(T,u,v-u) \leq 0 ,
\end{equation}
for every $v \in \PC^\T_\U$, where $w$, $w^0$ are defined in Proposition~\ref{propinputoutput}.
\end{lemma}

\begin{proof}
We denote by~$\T = \{ t_i \}_{i=0,\ldots,N}$ and, for every $v \in \PC^\T_\U$, we use the notation~$z_v \in \AC([0,T],\R)$ introduced in the proof of Lemma~\ref{lemcharactweakextremal}. Assume that~$(p,p^0)$ is a $\T$-averaged weak extremal lift of~$(x,u)$. Let $v \in  \PC^\T_\U$. We infer from~\eqref{eqAHG} that~$z_v(t_{i+1}) - z_v(t_i) \leq 0$ for every $i \in \{ 0,\ldots,N-1 \}$. With a telescoping sum we deduce that~$z_v(T) \leq 0$, which gives the first part of the proof. Conversely, we infer from~\eqref{ineq02} that
$$
\int_0^T \langle \nabla_u H (x(s),u(s),p(s),p^0,s) , v(s)-u(s) \rangle_{m} \; ds 
= \int_0^T \dot{z}_v (s) \; ds = z_v(T)-z_v(0) = z_v (T) \leq 0
$$
for every $v \in \PC^\T_\U $. Then, for any $\omega \in \U$ and any~$i \in \{ 0,\ldots,N-1 \}$, we take~$v \in \PC^\T_\U$ given by
$$ v(t) := \left\lbrace \begin{array}{lcl}
\omega & \text{if} & t \in [t_i,t_{i+1}), \\[3pt]
u(t) & \text{if} & t \notin [t_i,t_{i+1}),
\end{array}
\right. $$
for almost every $t \in [0,T]$. We thus obtain~\eqref{eqAHG}. 
\end{proof}

\section{Proofs of Filippov-type results}\label{appfilippov}
This section is dedicated to the statement of a Filippov-type lemma (Lemma~\ref{lemtech1} below), a technical argument that will be used several times. For the reader's convenience, we recall the proof of the well known Filippov theorem mentioned in Section~\ref{secOCP}.

\subsection{A technical Filippov-type lemma}\label{apptechnicallemma}

Let $d \in \N^*$ and let~$F : \R^n \times \R^m \times [0,T] \to \R^d$ be a continuous function, differentiable with respect to its first variable and such that~$\nabla_x F$ is continuous. Let $\Lambda$ be a closed subset of~$\R^d$ containing $0_{\R^d}$. We introduce the convexity assumption
\begin{equation}\label{eqhconv3}\tag{$\mathrm{H}^{\textrm{conv}}_{F}$}
\forall (x,t) \in \R^n \times [0,T], \; \VV_{F} (x,t) \text{ is convex}, 
\end{equation}
where
$$ \VV_{F} (x,t) := \{ ( f(x,u,t) , F(x,u,t) + \lambda ) \ \mid\ (u,\lambda) \in \U \times \Lambda \} $$
for all $(x,t) \in \R^n \times [0,T]$.

\begin{lemma}\label{lemtech1}
Assume that $\AA \neq \emptyset$. Under \eqref{eqhcomp} and~\eqref{eqhconv3}, every sequence~$(u_k)_{k \in \N}$ in $\UUadm$ has a subsequence (that we do not relabel) such that:
\begin{enumerate}
\item[\rm{(i)}] $x(\cdot,u_k)$ converges uniformly on $[0,T]$ to $x(\cdot,\overline{u})$;
\item[\rm{(ii)}] $f(x(\cdot,u_k),u_k,\cdot)$ converges weakly-star in $\L^\infty([0,T],\R^n)$ to $f(x(\cdot,\overline{u}),\overline{u},\cdot)$;
\item[\rm{(iii)}] $F(x(\cdot,u_k),u_k,\cdot)$ converges weakly-star in $\L^\infty([0,T],\R^d)$ to $F(x(\cdot,\overline{u}),\overline{u},\cdot) + \overline{\lambda}$;
\end{enumerate}
for some $\overline{u} \in \UUadm$ and some $\overline{\lambda} \in \L^\infty([0,T],\Lambda)$.
\end{lemma}

\begin{proof}
Consider a sequence $(u_k)_{k \in \N}$ in $\UUadm$. By~\eqref{eqhcomp}, there exists~$R > 0$ such that~$\Vert x(\cdot,u_k) \Vert_{\C} \leq R$ and~$\Vert u_k \Vert_{\L^\infty} \leq R$ for every~$k \in \N $. The sequence~$(f(x(\cdot,u_k),u_k,\cdot) , F(x(\cdot,u_k),u_k,\cdot))_{k \in \N}$ is bounded in~$\L^\infty([0,T],\R^{n+d})$ and thus (up to a subsequence that we do not relabel) converges weakly-star  in $\L^\infty([0,T],\R^{n+d})$ to some~$(g_1,g_2) \in \L^\infty([0,T],\R^{n+d})$. It follows from the Duhamel formula that the sequence $(x(\cdot,u_k))_{k \in \N}$ converges pointwisely on $[0,T]$ to the function $\overline{x} \in \AC$ defined by~$\overline{x}(t) := x_0 + \int_0^t g_1(s) \, ds$ for every $t \in [0,T]$, and thus uniformly by equi-Lipschitz continuity (as in \cite[Lemma~3]{BT2017}).
In particular $\overline{x}(T) = x_T$ and~$\Vert \overline{x} \Vert_{\C} \leq R$. 
We now define
$$
\HH := \{ h \in \L^2([0,T],\R^{n+d})\ \mid \ 
h(t) \in \VV_F (\overline{x}(t),t) \text{ for a.e.\ } t \in [0,T] \}. 
$$
Since $\U$ is compact and $\Lambda$ is closed, $\VV_F(x,t)$ is a closed (convex) subset of~$\R^{n+d}$ for all~$(x,t) \in \R^n \times [0,T]$. It follows from the partial converse of the Lebesgue dominated convergence theorem that~$\HH$ is a closed convex (and thus weakly closed) subset of~$\L^2([0,T],\R^{n+d})$. Since~$0_{\R^d} \in \Lambda$, the sequence~$(f(\overline{x},u_k,\cdot) , F(\overline{x},u_k,\cdot))_{k \in \N}$ belongs to~$\HH$ and is bounded in $\L^2([0,T],\R^{n+d})$. Thus (up to a subsequence that we do not relabel) it converges weakly in $\L^2([0,T],\R^{n+d})$ to some function~$(\overline{g}_1,\overline{g}_2)\in \L^2([0,T],\R^{n+d})$ which belongs to $\HH$. Since $\nabla_x f$ and $\nabla_x F$ are bounded on the compact set~$\BB_{\R^n}(0_{\R^n},R) \times \BB_{\R^m}(0_{\R^m},R) \times [0,T]$, there exists $M \geq 0$ such that
\begin{multline*}
\Vert (f(x(t,u_k),u_k(t),t) , F(x(t,u_k),u_k(t),t)) 
- (f(\overline{x}(t),u_k(t),t) , F(\overline{x}(t),u_k(t),t)) \Vert_{\R^{n+d}} \\[3pt]
 \leq M \Vert x(t,u_k) - \overline{x}(t) \Vert_{\R^n} 
\end{multline*}
for almost every $t \in [0,T]$ and all $k \in \N$. Using the above inequality and the Lebesgue dominated convergence theorem, the sequence
$$
\Big( (f(x(\cdot,u_k),u_k,\cdot) , F(x(\cdot,u_k),u_k,\cdot)) 
- (f(\overline{x},u_k,\cdot) , F(\overline{x},u_k,\cdot)) \Big)_{k \in \N}
$$
converges in $\L^2([0,T],\R^{n+d})$ to the null function. Since it also weakly converges in~$\L^2([0,T],\R^{n+d})$ to $(g_1,g_2) - (\overline{g}_1,\overline{g}_2)$, we conclude that $(g_1,g_2) = (\overline{g}_1,\overline{g}_2) \in \HH$. By the measurable selection theorem~\cite[Theorem~7.1]{himmelberg1975}, since $\U$ and $\Lambda$ are closed, there exist two measurable functions~$\overline{u} : [0,T] \to \U$ and~$\overline{\lambda} : [0,T] \to \Lambda$ such that
$$ (g_1(t),g_2(t)) = (f(\overline{x}(t),\overline{u}(t),t),F (\overline{x}(t),\overline{u}(t),t) + \overline{\lambda}(t) ) $$
for almost every $t \in [0,T]$. Since $\U$ is bounded, we get that~$\overline{u} \in \L^\infty_\U$ and we infer from the Duhamel formula that~$\overline{x} = x(\cdot,\overline{u}) $ and thus $\overline{u} \in \UUadm$. Moreover $\overline{\lambda} = g_2 - F(\overline{x},\overline{u},\cdot) \in \L^\infty([0,T],\Lambda)$. The proof is complete. 
\end{proof}

\subsection{Proof of the Filippov theorem mentioned in Section~\ref{secOCP}}
Since $\AA \neq \emptyset$, we have~$\UUadm \neq \emptyset$ (see Remark~\ref{rem587}). The infimum value of Problem~\eqref{theproblem} is~$\inf_{u \in \UUadm} \KK(u)$. Consider a minimizing sequence~$(u_k)_{k \in \N} \subset \UUadm$. By Lemma~\ref{lemtech1} applied with~$d = 1$, $F = L$ and~$\Lambda = \R_+$, $(u_k)_{k \in \N}$ has a subsequence (that we do not relabel) such that~$L(x(\cdot,u_k),u_k,\cdot)$ converges weakly-star in $\L^\infty([0,T],\R)$ to $L(x(\cdot,\overline{u}),\overline{u},\cdot) + \overline{\gamma}$, where~$\overline{u} \in \UUadm$ and~$\overline{\gamma} \in \L^\infty([0,T],\R_+)$. Hence $\inf_{u \in \UUadm} \KK(u) = \lim_{k \to \infty} \KK(u_k) = \int_0^T (L(x(s,\overline{u}),\overline{u}(s),s) + \overline{\gamma}(s)) \, ds \geq \KK(\overline{u})$, which concludes the proof. 

\section{Proofs of Proposition~\ref{propsampledfilippov} and Theorem~\ref{thmcv}}\label{appmainresults}
We use the notations and results of Appendices~\ref{app1}, \ref{appinput} and~\ref{appfilippov}. We set $\UU^\T := \UU \cap \PC^\T$ and $\UUadm^\T := \UUadm \cap \PC^\T$ for every~$\T \in \PP$. As to Remark~\ref{rem587}, given any $\T \in \PP$, we have~$\AA_\T \neq \emptyset$ if and only if~$\UUadm^\T \neq \emptyset$.

\subsection{Proof of Proposition~\ref{propsampledfilippov}}
Since $\AA_\T \neq \emptyset$ for some~$\T \in \PP$, we have $\UUadm^\T \neq \emptyset$. The infimum value of Problem~\eqref{thesampledproblem} is~$\inf_{u \in \UUadm^\T} \KK(u)$. Consider a minimizing sequence~$(u_k)_{k \in \N} \subset \UUadm^\T$. By~\eqref{eqhcomp}, there exists $R > 0$ such that $\Vert x(\cdot,u_k) \Vert_{\C} \leq R$ and~$\Vert u_k \Vert_{\L^\infty} \leq R$ for every $k \in \N$. By Remark~\ref{remtruncated}, $x(\cdot,u_k) = x^R(\cdot,u_k) = \E^R(u_k)$ for every~$k \in \N$. Since~$\U $ is compact and~$\PC^\T$ is a finite-dimensional space, there exists a subsequence (that we do not relabel) such that $(u_k)_{k \in \N}$ converges in~$\L^\infty$ (and thus in~$\L^1$) to some~$\overline{u} \in \PC^\T_{\U}$ which moreover satisfies~$\Vert \overline{u} \Vert_{\L^\infty} \leq R$. By continuity of~$\E^R$ in $\L^1$-norm (see Proposition~\ref{propinputoutputR}), we get that~$\Vert x^R(\cdot,\overline{u}) \Vert_{\C} \leq R$. Using again Remark~\ref{remtruncated}, we get that~$\overline{u} \in \UU$ and $x(\cdot,\overline{u}) = x^R(\cdot,\overline{u})$. Similarly, by continuity of $\E^R_T$ in $\L^1$-norm and since~$x_T = x(T,u_k) = x^R(T,u_k) = \E^R_T(u_k)$ for every~$k \in \N$, we infer that~$x_T = \E^R_T(\overline{u}) = x^R(T,\overline{u}) = x(T,\overline{u})$. At this step we have proved that~$\overline{u} \in \UUadm^\T$. To conclude, we recall that~$\lim_{k \to \infty} \KK(u_k) = \inf_{u \in \UUadm^\T} \KK(u)$ and we use the facts that $ \KK(\overline{u}) = \KK^R(\overline{u})$ and $\KK (u_k) = \KK^R(u_k)$ for every~$k \in \N$ and the continuity of $\KK^R$ in $\L^1$-norm. Precisely, we write~$\inf_{u \in \UUadm^\T} \KK(u) = \lim_{k \to \infty} \KK(u_k) = \lim_{k \to \infty} \KK^R(u_k) = \KK^R(\overline{u}) = \KK(\overline{u})$. The proof is complete.

\begin{remark}
By the above proof, one can note that, actually, Proposition~\ref{propsampledfilippov} remains valid if the assumption \eqref{eqhcomp} is weakened to:
$$ \exists R > 0, \; \forall (x,u) \in \AA_\T, \; \Vert x \Vert_\C + \Vert u \Vert_{\L^\infty} \leq R . $$
\end{remark}

\subsection{Proof of the first part of Theorem~\ref{thmcv}}\label{appfaible}

By Lemma~\ref{lemmaaccessible}, there exists $\delta > 0$ such that~$\AA_\T \neq \emptyset$ for every~$\T \in \PP$ satisfying~$\Vert \T \Vert \leq \delta$. By~\eqref{eqhcomp} and Proposition~\ref{propsampledfilippov}, Problem~\eqref{thesampledproblem} has at least one solution for every~$\T \in \PP$ satisfying~$\Vert \T \Vert \leq \delta$. Let~$(x^*_\T,u^*_\T)$ be a family of solutions to~\eqref{thesampledproblem} for every~$\T \in \PP$ satisfying~$\Vert \T \Vert \leq \delta$.

By a standard argument of unique closure point, it suffices to prove that~(i) and~(ii) are satisfied for at least one subsequence of any sequence~$(\T_k)_{k \in \N}$ in $\PP$ satisfying~$\Vert \T_k \Vert_{\PP} \leq \delta$ for every~$k \in \N$ and~$\lim_{k \to \infty} \Vert \T_k \Vert = 0$. Let~$(\T_k)_{k \in \N}$ be such a sequence.
In what follows, for the ease of notations, we denote~$u^*_k := u^*_{\T_k}$ and~$x^*_k := x^*_{\T_k} = x(\cdot,u^*_k)$ for every~$k \in \N$. By Lemma~\ref{lemtech1} applied with $d = 1$, $F=L$ and $\Lambda = \R_+$, we get that~$(u^*_k)_{k \in \N}$ has a subsequence (that we do not relabel) such that $x^*_k$ converges uniformly to~$x(\cdot,\overline{u})$ on $[0,T]$, $f(x^*_k,u^*_k,\cdot)$ converges weakly-star in~$\L^\infty([0,T],\R^n)$ to $f(x(\cdot,\overline{u}),\overline{u},\cdot)$ and $L(x^*_k,u^*_k,\cdot)$ converges weakly-star in~$\L^\infty([0,T],\R)$ to $L(x(\cdot,\overline{u}),\overline{u},\cdot) + \overline{\gamma}$, where~$\overline{u} \in \UUadm$ and~$\overline{\gamma} \in \L^\infty([0,T],\R_+)$. 

It suffices to prove that~$\overline{u} = u^*$ and that~$\overline{\gamma}$ is the null function. Let us prove that~$\overline{u}$ is a solution to~\eqref{theproblem}. First, by~\eqref{eqhcomp}, there exists~$R >0$ such that~$\Vert x(\cdot,u ) \Vert_\C \leq R$ and~$\Vert u \Vert_{\L^\infty} \leq R$ for every~$u \in \UUadm$. Consider a sequence~$(\rho_k)_{k \in \N}$ of positive real numbers converging to zero, and consider the corresponding sequence~$(\delta_k)_{k \in \N}$ of positive real numbers provided in Lemma~\ref{lemmaaccessible} associated with~$(x^*,u^*) \in \AA$ which does not have any abnormal strong extremal lift. Since~$\lim_{k \to \infty} \Vert \T_k \Vert = 0$, up to a subsequence (that we do not relabel), we have~$\Vert \T_k \Vert \leq \delta_k$ for every~$k \in \N$. By Lemma~\ref{lemmaaccessible}, there exists~$v_k \in \UUadm^{\T_k} $ such that~$\Vert v_k - u^* \Vert_{\L^1} \leq \rho_k$ for every $k \in \N$. By optimality of~$u^*$ and~$u^*_k$, we have~$ \KK (u^*) \leq \KK (u^*_k) \leq \KK (v_k) = \KK^R (v_k)$ for every~$k \in \N$. The latter equality follows from~$\Vert x(\cdot,v_k) \Vert_\C \leq R$ and~$\Vert v_k \Vert_{\L^\infty} \leq R$ since~$v_k \in \UUadm^{\T_k} \subset \UUadm$ for every~$k \in \N$. By continuity of~$\KK^R$ in~$\L^1$-norm (see Proposition~\ref{propinputoutputR}), we get by taking the limit that~$ \KK (u^*) \leq  \int_0^T  (L(x(s,\overline{u}),\overline{u}(s),\cdot) + \overline{\gamma}(s)) \, ds \leq \KK^R(u^*) = \KK (u^*)$. The latter equality follows from~$\Vert x(\cdot,u^*) \Vert_\C \leq R$ and~$\Vert u^* \Vert_{\L^\infty} \leq R$ since~$u^* \in \UUadm$. We finally get that~$\KK (\overline{u}) = \int_0^T  L(x(s,\overline{u}),\overline{u}(s),s) \, ds \leq \int_0^T ( L(x(s,\overline{u}),\overline{u}(s),\cdot) + \overline{\gamma}(s) ) \, ds \leq \KK (u^*)$. By optimality of~$u^* \in \UUadm$ and since $\overline{u} \in \UUadm$, we infer that~$\overline{u}$ is a solution to~\eqref{theproblem} and thus~$\overline{u} = u^*$ by uniqueness. Moreover we  have also proved that~$\overline{\gamma}$ is the null function. 

\subsection{Proof of the second part of Theorem~\ref{thmcv}}\label{appforte}

By the Pontryagin maximum principle obtained in~\cite{BT2015,BT2016}, the solution~$(x^*_\T,u^*_\T)$ to~\eqref{thesampledproblem} has a $\T$-averaged weak extremal lift~$(p_\T,p^0_\T)$ for every~$\T \in \PP$ satisfying~$\Vert \T \Vert \leq \delta$. The proof is divided in three steps.

\paragraph{First step} Let us prove that there exists~$0 < \delta' \leq \delta$ such that~$p^0_{\T} \neq 0$ for every~$\T \in \PP$ satisfying~$\Vert \T \Vert \leq \delta'$. By contradiction, assume that there exists a sequence~$(\T_k)_{k \in \N}$ in~$\PP$ satisfying~$\Vert \T_k \Vert \leq \delta$ for every~$k \in \N$ and~$\lim_{k \to \infty} \Vert \T_k \Vert = 0$, such that~$p^0_{\T_k} = 0$ for every~$k \in \N$. To get a contradiction we will prove that $(x^*,u^*)$ has an abnormal weak extremal lift.

In what follows, for the ease of notations, we denote~$u^*_k := u^*_{\T_k}$,~$x^*_k := x^*_{\T_k} = x(\cdot,u^*_k)$,~$p_k := p_{\T_k}$ and~$p^0_k := p^0_{\T_k}$ for every~$k \in \N$. By Lemma~\ref{lemtech1} applied with~$d = (n \times n) + (n \times m) + n + 1 + n + m + 1$, $F=F_\nabla$ and~$\Lambda = \Gamma$, we get that~$(u^*_k)_{k \in \N}$ has a subsequence (that we do not relabel) such that~$x^*_k$ converges uniformly on~$[0,T]$ to~$x(\cdot,\overline{u})$,~$f(x^*_k,u^*_k,\cdot)$ converges weakly-star in~$\L^\infty([0,T],\R^n)$ to~$f(x(\cdot,\overline{u}),\overline{u},\cdot)$ and~$F_\nabla(x^*_k,u^*_k,\cdot)$ converges weakly-star in~$\L^\infty([0,T],\R^d)$ to~$F_\nabla(x(\cdot,\overline{u}),\overline{u},\cdot) + \overline{\gamma}$, where~$\overline{u} \in \UUadm$ and~$\overline{\gamma} \in \L^\infty([0,T],\Gamma)$. As in the proof of the first part of Theorem~\ref{thmcv}, we prove that~$\overline{u} = u^*$ (and that the fourth component~$\overline{\gamma}_4$ of~$\overline{\gamma}$ is the null function). 

Let $v \in \L^\infty_\U$ and $v_k := v^{\T_k} \in \PC^{\T_k}_\U $ for every~$k \in \N$ (see~\eqref{eqaveraged} and Lemma~\ref{lem2}) and recall that the sequence~$(v_k)_{k \in \N}$ converges to $v$ in~$\L^1$ (see Lemma~\ref{lem1}). Since $p^0_{k} = 0$, it follows from Lemma~\ref{lemcharactsampledextremal} that
$$ \langle p_k(T) , w(T,u^*_k,v_k-u^*_k)\rangle_{n} \leq 0  $$
for every~$k \in \N$, where~$w$ is defined in Proposition~\ref{propinputoutput}. Since~$p^0_{k} = 0$ we have that~$p_k(T) \neq 0_{\R^n}$ for every~$k \in \N$ (see Remark~\ref{rempermanent}). Up to a subsequence (that we do not relabel), the sequence~$(\frac{p_k(T)}{\Vert p_k(T) \Vert_{\R^n}})_{k \in \N}$ converges to some~$\theta \in \R^n \bs \{ 0_{\R^n} \}$. Thus, dividing the above inequality by~$\Vert p_k(T) \Vert_{\R^n}$ and passing to the limit (using in particular the weak-star convergences and Proposition~\ref{propconvlineardiffeq}), we get that 
$$ \langle \theta , w(T,u^*,v-u^*) \rangle_{n} \leq 0. $$ 
Defining $q^0 := 0$ and $q(t) := \Phi_{u^*}(T,t)^\top \theta$ for every~$t \in [0,T]$ (where~$\Phi_{u^*}(\cdot,\cdot)$ is the state-transition matrix of~$\nabla_x f (x^*,u^*,\cdot) \in \L^\infty([0,T],\R^{n \times n})$), we get that~$(q,q^0) \in \AC \times \R_-$ is a nontrivial pair satisfying the adjoint equation~\eqref{eqAE} associated with~$(x^*,u^*)$. Since the latter inequality is satisfied for every~$v \in \L^\infty_\U$, we get from Lemma~\ref{lemcharactweakextremal} that~$(q,q^0)$ is an abnormal weak extremal lift of~$(x^*,u^*)$, which raises a contradiction.

\paragraph{Second step} By the first step, we renormalize the extremal lifts so that $p^0 = p^0_{\T} = -1$ for every~$\T \in \PP$ satisfying~$\Vert \T \Vert \leq \delta'$. Let us prove that there exists~$0 < \delta'' \leq \delta'$ such that~$\Vert p_\T (T) \Vert_{\R^n}$ is bounded for every~$\T \in \PP$ satisfying~$\Vert \T \Vert \leq \delta''$. By contradiction, assume that there exists a sequence~$(\T_k)_{k \in \N}$ in $\PP$ satisfying~$\Vert \T_k \Vert \leq \delta'$ for every~$k \in \N$ and~$\lim_{k \to \infty} \Vert \T_k \Vert = 0$, such that~$\lim_{k \to \infty} \Vert p_{\T_k} (T) \Vert_{\R^n} = +\infty$. Without loss of generality we assume that $p_{\T_k} (T) \neq 0_{\R^n}$ for every~$k \in \N$. To raise a contradiction we will prove that~$(x^*,u^*)$ has an abnormal weak extremal lift.

In what follows, for the ease of notations, we denote~$u^*_k := u^*_{\T_k}$,~$x^*_k := x^*_{\T_k} = x(\cdot,u^*_k)$,~$p_k := p_{\T_k}$ and~$p^0_k := p^0_{\T_k}$ for every~$k \in \N$. By Lemma~\ref{lemtech1} applied with~$d = (n \times n) + (n \times m) + n + 1 + n + m + 1$, $F=F_\nabla$ and~$\Lambda = \Gamma$, we get that~$(u^*_k)_{k \in \N}$ has a subsequence (that we do not relabel) such that~$x^*_k$ converges uniformly on~$[0,T]$ to~$x(\cdot,\overline{u})$,~$f(x^*_k,u^*_k,\cdot)$ converges weakly-star in~$\L^\infty([0,T],\R^n)$ to~$f(x(\cdot,\overline{u}),\overline{u},\cdot)$ and~$F_\nabla(x^*_k,u^*_k,\cdot)$ converges weakly-star in~$\L^\infty([0,T],\R^d)$ to~$F_\nabla(x(\cdot,\overline{u}),\overline{u},\cdot) + \overline{\gamma}$, where~$\overline{u} \in \UUadm$ and~$\overline{\gamma} \in \L^\infty([0,T],\Gamma)$. As in the proof of the first part of Theorem~\ref{thmcv}, we prove that~$\overline{u} = u^*$ (and that the fourth component~$\overline{\gamma}_4$ of~$\overline{\gamma}$ is the null function). 

As in the first step, let $v \in \L^\infty_\U$ and $v_k := v^{\T_k} \in \PC^{\T_k}_\U $ for every~$k \in \N$. Recalling that $p^0_{k} = -1$, we infer from Lemma~\ref{lemcharactsampledextremal} that
$$ \langle p_k(T) , w(T,u^*_k,v_k-u^*_k) \rangle_{n} - w^0(T,u^*_k,v_k-u^*_k) \leq 0  $$
for every~$k \in \N$, where $w$, $w^0$ are defined in Proposition~\ref{propinputoutput}. Up to a subsequence (that we do not relabel), the sequence~$(\frac{p_k(T)}{\Vert p_k(T) \Vert_{\R^n}})_{k \in \N}$ converges to some~$\theta \in \R^n \bs \{ 0_{\R^n} \}$. Dividing the latter inequality by~$\Vert p_k(T) \Vert_{\R^n}$ (which converges to~$+\infty$) and taking the limit (using in particular the weak-star convergences and Proposition~\ref{propconvlineardiffeq}), we get that 
$$ \langle \theta , w(T,u^*,v-u^*) \rangle_{n} \leq 0. $$
We have used the fact that~$w^0(T,u^*_k,v_k-u^*_k)$ converges to~$w^0(T,u^*,v-u^*)-\int_0^T \overline{\gamma}_7(s) \; ds$, where $\overline{\gamma}_7$ is the seventh component of~$\overline{\gamma}$. This convergence easily follows from the weak-star convergences. Finally, since the latter inequality is satisfied for every~$v \in \L^\infty_\U$, we get a contradiction as in the first step.

\paragraph{Third step} 
As in the proof of the first part of Theorem~\ref{thmcv}, it suffices to prove that~(iii) is satisfied for at least one subsequence of any sequence~$(\T_k)_{k \in \N}$ in $\PP$ satisfying~$\Vert \T_k \Vert_{\PP} \leq \delta''$ for every~$k \in \N$ and~$\lim_{k \to \infty} \Vert \T_k \Vert = 0$. Let~$(\T_k)_{k \in \N}$ be such a sequence.
In what follows, for the ease of notations, we denote~$u^*_k := u^*_{\T_k}$,~$x^*_k := x^*_{\T_k} = x(\cdot,u^*_k)$,~$p_k := p_{\T_k}$ and~$p^0_k := p^0_{\T_k}$ for every~$k \in \N$. By Lemma~\ref{lemtech1} applied with~$d = (n \times n) + (n \times m) + n + 1 + n + m + 1$, $F=F_\nabla$ and~$\Lambda = \Gamma$, we get that~$(u^*_k)_{k \in \N}$ has a subsequence (that we do not relabel) such that~$x^*_k$ converges uniformly on~$[0,T]$ to~$x(\cdot,\overline{u})$,~$f(x^*_k,u^*_k,\cdot)$ converges weakly-star in~$\L^\infty([0,T],\R^n)$ to~$f(x(\cdot,\overline{u}),\overline{u},\cdot)$ and~$F_\nabla(x^*_k,u^*_k,\cdot)$ converges weakly-star in~$\L^\infty([0,T],\R^d)$ to~$F_\nabla(x(\cdot,\overline{u}),\overline{u},\cdot) + \overline{\gamma}$, where~$\overline{u} \in \UUadm$ and~$\overline{\gamma} \in \L^\infty([0,T],\Gamma)$. As in the proof of the first part of Theorem~\ref{thmcv}, we prove that~$\overline{u} = u^*$ (and that the fourth component~$\overline{\gamma}_4$ of~$\overline{\gamma}$ is the null function).

From the first and second steps, $p^0_{k}=-1$ for every~$k \in \N$ and the sequence~$(p_k(T))_{k \in \N}$ converges, up to a subsequence (that we do not relabel), to some~$\Psi \in \R^n$. By the weak-star convergences and the backward version of Proposition~\ref{propconvlineardiffeq}, the sequence~$(p_k)_{k \in \N}$ converges uniformly on~$[0,T]$ to the function $q \in \AC([0,T],\R^n)$ that is the unique solution to the (backward) linear Cauchy problem given by
\begin{equation*}
\left\lbrace
\begin{array}{ll}
\!\! \dot{q}(t) = - \nabla_x f(x^*(t),u^*(t),t)^\top q(t)  + \nabla_x L(x^*(t),u^*(t),t)  \\[3pt]
\!\! q(T) = \Psi. &
\end{array}
\right.
\end{equation*}
Defining~$q^0 = -1$, we get that~$(q,q^0) \in \AC \times \R_-$ is a nontrivial pair satisfying the adjoint equation~\eqref{eqAE} associated with~$(x^*,u^*)$.

Let us prove that~$(q,q^0)$ is a weak extremal lift of~$(x^*,u^*)$. As in the first and second steps, let~$v \in \L^\infty_\U$ and $v_k := v^{\T_k} \in \PC^{\T_k}_\U $ for every~$k \in \N$. Recalling that $p^0_{k} = -1$, we infer from Lemma~\ref{lemcharactsampledextremal} that
$$ \langle p_k(T) , w(T,u^*_k,v_k-u^*_k) \rangle_{n} - w^0(T,u^*_k,v_k-u^*_k) \leq 0  $$
for every~$k \in \N$, where $w$, $w^0$ are defined in Proposition~\ref{propinputoutput}. Taking the limit in the above inequality (using in particular the weak-star convergences and Proposition~\ref{propconvlineardiffeq}), we get that 
\begin{multline*}
\langle q(T) , w(T,u^*,v-u^*) \rangle_{n} + q^0 w^0(T,u^*,v-u^*) \\[3pt]
\leq \langle q(T) , w(T,u^*,v-u^*) \rangle_{n} 
+ q^0 \bigg( w^0(T,u^*,v-u^*) -\int_0^T \overline{\gamma}_7(s) \; ds \bigg) \leq 0. $$
\end{multline*}
As in the second step, we have used the fact that~$w^0(T,u^*_k,v_k-u^*_k)$ converges to~$w^0(T,u^*,v-u^*)-\int_0^T \overline{\gamma}_7(s) \; ds$, where $\overline{\gamma}_7$ is the seventh component of~$\overline{\gamma}$ which has nonnegative values. Since the latter inequality is satisfied for every~$v \in \L^\infty_\U$, we infer from Lemma~\ref{lemcharactweakextremal} that~$(q,q^0)$ is a weak extremal lift of~$(x^*,u^*)$. By uniqueness we get that~$q=p$.



\bibliographystyle{plain}
\bibliography{bibIEEEbourdintrelat}

\begin{thebibliography}{10}

\bibitem{acker}
J.~E. Ackermann.
\newblock {\em Sampled-data control. Volume 1}.
\newblock Springer-Verlag, Berlin-New-York, 1983.

\bibitem{azhm}
V.~Azhmyakov, M.~Basin, and C.~Reincke-Collon.
\newblock Optimal {LQ}-type switched control design for a class of linear
  systems with piecewise constant inputs.
\newblock In {\em Proceedings of the 19th World Congress The International
  Federation of Automatic Control}, 2014.

\bibitem{bini2009}
E.~Bini.
\newblock {\em Design of optimal control systems}.
\newblock PhD thesis, University of Pisa, Italy, 2009.

\bibitem{bini2014}
E.~Bini and G.~Buttazzo.
\newblock The optimal sampling pattern for linear control systems.
\newblock {\em IEEE Trans. Automat. Control}, 59(1):78--90, 2014.

\bibitem{BonnardChyba}
B.~Bonnard and M.~Chyba.
\newblock {\em Singular trajectories and their role in control theory},
  volume~40 of {\em Math\'{e}matiques \& Applications (Berlin) [Mathematics \&
  Applications]}.
\newblock Springer-Verlag, Berlin, 2003.

\bibitem{BT2013}
L.~Bourdin and E.~Tr\'{e}lat.
\newblock Pontryagin maximum principle for finite dimensional nonlinear optimal
  control problems on time scales.
\newblock {\em SIAM J. Control Optim.}, 51(5):3781--3813, 2013.

\bibitem{BT2015}
L.~Bourdin and E.~Tr\'elat.
\newblock Pontryagin maximum principle for optimal sampled-data control
  problems.
\newblock In {\em Proceedings of the IFAC workshop CAO}, 2015.

\bibitem{BT2016}
L.~Bourdin and E.~Tr\'{e}lat.
\newblock Optimal sampled-data control, and generalizations on time scales.
\newblock {\em Math. Control Relat. Fields}, 6(1):53--94, 2016.

\bibitem{BT2017}
L.~Bourdin and E.~Tr\'{e}lat.
\newblock Linear-quadratic optimal sampled-data control problems: convergence
  result and {R}iccati theory.
\newblock {\em Automatica J. IFAC}, 79:273--281, 2017.

\bibitem{BT2021}
L.~Bourdin and E.~Tr\'{e}lat.
\newblock Robustness under control sampling of reachability in fixed time for
  nonlinear control systems.
\newblock {\em Mathematics of Control, Signals, and Systems}, 33:515–551,
  2021.

\bibitem{BressanPiccoli}
A.~Bressan and B.~Piccoli.
\newblock {\em Introduction to the mathematical theory of control}, volume~2 of
  {\em AIMS Series on Applied Mathematics}.
\newblock American Institute of Mathematical Sciences (AIMS), Springfield, MO,
  2007.

\bibitem{Cesari}
L.~Cesari.
\newblock {\em Optimization -- {T}heory and applications}, volume~17 of {\em
  Applications of Mathematics (New York)}.
\newblock Springer-Verlag, New York, 1983.

\bibitem{chen}
T.~Chen and B.~Francis.
\newblock {\em Optimal sampled-data control systems}.
\newblock Communications and Control Engineering Series. Springer-Verlag
  London, Ltd., London, 1996.

\bibitem{ChitourJeanTrelat_JDG2006}
Y.~Chitour, F.~Jean, and E.~Tr\'{e}lat.
\newblock Genericity results for singular curves.
\newblock {\em J. Differential Geom.}, 73(1):45--73, 2006.

\bibitem{ChitourJeanTrelat_SICON2008}
Y.~Chitour, F.~Jean, and E.~Tr\'{e}lat.
\newblock Singular trajectories of control-affine systems.
\newblock {\em SIAM J. Control Optim.}, 47(2):1078--1095, 2008.

\bibitem{fada}
S.~Fadali and A.~Visioli.
\newblock {\em Digital control Engineering. Analysis and design}.
\newblock Elsevier, 2013.

\bibitem{filippov1959}
A.~F. Filippov.
\newblock On some questions in the theory of optimal regulation: existence of a
  solution of the problem of optimal regulation in the class of bounded
  measurable functions.
\newblock {\em Vestnik Moskov. Univ. Ser. Mat. Meh. Astr. Fiz. Him.},
  1959(2):25--32, 1959.

\bibitem{HaberkornTrelat_SICON2011}
T.~Haberkorn and E.~Tr\'{e}lat.
\newblock Convergence results for smooth regularizations of hybrid nonlinear
  optimal control problems.
\newblock {\em SIAM J. Control Optim.}, 49(4):1498--1522, 2011.

\bibitem{himmelberg1975}
C.~J. Himmelberg.
\newblock Measurable relations.
\newblock {\em Fund. Math.}, 87:53--72, 1975.

\bibitem{Iser}
R.~Isermann.
\newblock {\em Digital control systems. {V}ol. 1}.
\newblock Springer-Verlag, Berlin, second edition, 1989.

\bibitem{land}
I.~D. Landau.
\newblock {\em Digital Control Systems}.
\newblock Springer, 2006.

\bibitem{LeeMarkus}
E.~B. Lee and L.~Markus.
\newblock {\em Foundations of optimal control theory}.
\newblock John Wiley \& Sons, Inc., New York-London-Sydney, 1967.

\bibitem{levis1971}
A.~Levis, R.~Schlueter, and M.~Athans.
\newblock On the behavior of optimal linear sampled-data regulators.
\newblock {\em International Journal of Control}, 13:343--361, 1971.

\bibitem{nesi}
D.~Nesi\'{c} and A.~R. Teel.
\newblock Sampled-data control of nonlinear systems: an overview of recent
  results.
\newblock In {\em Perspectives in robust control ({N}ewcastle, 2000)}, volume
  268 of {\em Lect. Notes Control Inf. Sci.}, pages 221--239. Springer, London,
  2001.

\bibitem{pontryagin1962}
L.~S. Pontryagin, V.~G. Boltyanskii, R.~V. Gamkrelidze, and E.~F. Mishchenko.
\newblock {\em The mathematical theory of optimal processes}.
\newblock Interscience Publishers John Wiley \& Sons, Inc.\, New York-London,
  1962.

\bibitem{raga}
J.~R. Ragazzini.
\newblock {\em Sampled-data control systems}.
\newblock McGraw-Hill, 1958.

\bibitem{RiffordTrelat_MA2005}
L.~Rifford and E.~Tr\'{e}lat.
\newblock Morse-{S}ard type results in sub-{R}iemannian geometry.
\newblock {\em Math. Ann.}, 332(1):145--159, 2005.

\bibitem{RiffordTrelat_JEMS2009}
L.~Rifford and E.~Tr\'{e}lat.
\newblock On the stabilization problem for nonholonomic distributions.
\newblock {\em J. Eur. Math. Soc. (JEMS)}, 11(2):223--255, 2009.

\bibitem{SilvaTrelat_TAC2010}
C.~Silva and E.~Tr\'{e}lat.
\newblock Smooth regularization of bang-bang optimal control problems.
\newblock {\em IEEE Trans. Automat. Control}, 55(11):2488--2499, 2010.

\bibitem{souz}
M.~Souza, G.~W.~G. Vital, and J.~C. Geromel.
\newblock Optimal sampled-data state feedback control of linear systems.
\newblock In {\em Proceedings of the 19th World Congress The International
  Federation of Automatic Control}, 2014.

\bibitem{toiv}
H.~T. Toivonen and M.~F. Sagfors.
\newblock The sampled-data {$\mathrm{H}_\infty$} problem: a unified framework
  for discretization-based methods and {R}iccati equation solution.
\newblock {\em Internat. J. Control}, 66(2):289--309, 1997.

\bibitem{tou}
J.~T. Tou.
\newblock {\em Optimum design of digital control systems}.
\newblock Academic Press, New York-London, 1963.

\bibitem{trelat2005}
E.~Tr\'{e}lat.
\newblock {\em Contr\^{o}le optimal}.
\newblock Math\'{e}matiques Concr\`etes. Vuibert, Paris, 2005.

\bibitem{Trelat_JOTA2012}
E.~Tr\'{e}lat.
\newblock Optimal control and applications to aerospace: some results and
  challenges.
\newblock {\em J. Optim. Theory Appl.}, 154(3):713--758, 2012.

\end{thebibliography}

%
%

\end{document}